\documentclass[11pt]{amsart}
\usepackage{amssymb,amsmath,amsfonts,amsthm}
\usepackage{comment}
\usepackage{color}
\usepackage{enumitem}

\usepackage{mdframed}

\newtheorem{theorem}{Theorem}[section]
\newtheorem{lemma}[theorem]{Lemma}
\newtheorem{proposition}[theorem]{Proposition}

\newtheorem{ltheorem}{Theorem} 

\theoremstyle{definition}
\newtheorem{definition}[theorem]{Definition}

\newtheorem{remark}[theorem]{Remark}

\renewcommand{\phi}{\varphi}

\usepackage{mdframed}

\def\real{\mathbb{R}}

\def\integer{\mathbb{Z}}
\def\natural{\mathbb{N}}
\def\torus{\mathbb{T}^3}
\def\SL{SL(3,\integer)}

\def\supp{\operatorname{supp}}
\def\d{\operatorname{dist}}
\def\dim{\operatorname{dim}}

\def\leng{\operatorname{length}}
\def\vol{\operatorname{Vol}}

\def\cW{\mathcal{W}}

\def\cH{\mathcal{H}}
\def\ch{\mathfrak{h}}

\def\quand{\quad\text{and}\quad}

\newcommand{\norm}[1]{{\left\lVert  #1  \right\rVert}}
\newcommand{\abs}[1]{{\left\lvert  #1  \right\rvert}}

\title
{Geometric growth for Anosov maps on the $3$ torus}
\author{ Mauricio Poletti}
\address{LAGA -- Universit\'e Paris 13, 99 Av. Jean-Baptiste Cl\'ement, 93430 Villetaneus, France.}
\email{mpoletti@impa.br}

\begin{document}

\begin{abstract}
We prove that for Anosov maps of the $3$-torus if the Lyapunov exponents of absolutely continuous measures in every direction are equal to the geometric growth of the invariant foliations 
then $f$ is $C^1$ conjugated to its linear part.
\end{abstract}

\maketitle


\section{Introduction}
Let $f:M\to M$ be a $C^\nu$ diffeomorphism, $\nu>1$, of a compact $d$-dimensional $C^\nu$ manifold, we say that $\cW$ is a \emph{$C^\nu$ foliation} of dimension $k$ if the leaves are $k$-dimensional $C^\nu$ manifolds and varies continuously in the transversal direction.
To be more clear, for every $p\in M$ there exists a H\"older continuous map $\Phi:[-1,1]^d \to M$ such that
\begin{itemize}
\item $\Phi(0,0)=p$, for every $y\in [-1,1]^{d-k}$ $\Phi([-1,1]^k\times\{y\})\subset \cW(\Phi(0,y))$,
\item the restriction of $\Phi$ to $[-1,1]^k\times\{y\}$ is $C^\nu$,
\item $(x,y)\mapsto\frac{\partial \Phi}{\partial x}(x,y)$ is H\"older continuous. 
\end{itemize}
We call $\Phi([-1,1]^d)$ a \emph{foliated box}.

We call a foliation \emph{$f$-invariant} if $f$ maps leaves to leaves, i.e: $f(\cW(x))=\cW(f(x))$.

Fixing some Riemannian metric on $M$, we induce a Riemannian metric on $\cW(x)$ given by the restriction of the metric to the leaf. Lets call $\cW_r(x)$ the closed disc of radius $r$, centered at $x$, in $\cW(x)$. Denote by $\vol_\cW(x)$ the volume
induced by the Riemannian metric on $\cW(x)$.

We say that a foliation $\cW$ is \emph{expanding} if there exist $C>0$ and $\lambda>1$ such that $\norm{Df^n(x)v}\geq C\lambda^n \norm{v}$ for every $v\in T_x\cW$. We say that $\cW$ is contracting if it is expanding for $f^{-1}$.
Observe that after a change of metric we can always take $C=1$.

Following \cite{SX09}, we define the geometric growth of an invariant foliation $\cW$:
\begin{definition}
Let $\cW$ be an invariant expanding foliation. The \emph{geometric growth} of a ball of radius $r$ centered at $x$ is defined by
\begin{equation*}
\chi\left(x,r\right)=\limsup \frac{1}{n}\log \vol\left(f^{n}\left(\cW_{r}(x)\right)\right),
\end{equation*}
taking the supremum over $x$ we define 
\begin{equation*}
\chi^{\cW}_{f}|r=\sup_{x\in M} \chi\left(x,r\right).
\end{equation*}
For a contracting foliation we define his geometric growth as $-\chi^{\cW}_{f^{-1}}|r$.
\end{definition}

It is easy to see that $\chi^{\cW}_{f}|r$ does not depends on $r$ and the Riemannian metric (see for example \cite{SX09}), we call this value $\chi^{\cW}_{f}$ the \emph{geometric growth of $\cW$}.

We say that an invariant measure $\mu$ is \emph{absolutely continuous on $\cW$} if the disintegration of the measure with respect to this foliation is absolutely continuous
with respect to the volume measure in the leaves of $\cW$.

\begin{remark}
Usually the partition into leaves of $\cW$ is not measurable (see \cite{FET}), so in order to disintegrate $\mu$ we take a measurable partition such that every element of the partition is contained in some leaf of $\cW$. 

For example we can cover $M$ with foliated boxes and take the partition given by the connected components of the intersection of the leaves of $\cW$ with the foliated boxes.
\end{remark}

Given an $f$-invariant measure $\mu$ we define the Lyapunov exponent with respect to an invariant foliation $\cW$ as
$$
\lambda^\mu_\cW(x)=\lim_{n\to\infty}\frac{1}{n}\log \norm{Df^n(x)\mid_{T_x\cW}},
$$
by Oseledets theorem~\cite{Ose68} this limit exist for $\mu$-almost every point $x\in M$. Moreover, for almost every $x$, there exists an invariant decomposition 
$T_x\cW=E^1(x)\oplus \cdots \oplus E^k(x)$ and numbers $\lambda^\mu_\cW(x)=\lambda^1_\cW(x)>\cdots>\lambda^k_\cW(x)$ such that 
$$
\lambda^i_\cW(x)=\lim_{n\to\pm \infty}\frac{1}{n}\log \norm{Df^n v^i},\quad v^i\in E^i(x).
$$

Relations between the Lyapunov exponents and the geometric growth were studied by Saghin-Xia~\cite{SX09}, they proved that if the invariant measure $\mu$ is absolutely continuous 
on $\cW$ then $\Delta_{\cW}\leq \chi^{\cW}_{f}$, where $\Delta_{\cW}=\int \sum \dim E^i \lambda^i_\cW d\mu$. 

A natural question that arises is: 
\emph{What can we say about $f$ when we have the extremal case $\Delta_{\cW}= \chi^{\cW}_{f}$?}

We address this question proving a rigidity type result for Anosov maps on the torus $\torus=\real^3/\integer^3$. Let's describe the context in which we are going to work.

\subsection{Partially hyperbolic Anosov maps}
We say that $f:\mathbb{T}^d \to \mathbb{T}^d$ is an \emph{Anosov map} if there exists an invariant splitting 
$T_x M=E^u(x)\oplus E^s(x)$, such that $E^u$ is expanding and $E^s$ is contracting, by this we mean that  there exist  constants $\lambda>1$, $C>0$ such that
$$
\norm{Df^{-n}(x)v^u}\leq C\lambda^{-n}\quand \norm{Df^{n}(x)v^s}\leq C\lambda^{-n},$$
for every $n\in\natural$, and unitary vectors $v^u\in E^u$ and $v^s\in E^s$.

Let $F:\real^d\to\real^d$ be a $\integer^d$-equivariant lift of $f$ to the universal cover, the matrix $A=F(\cdot)-F(0):\integer^d\to \integer^d$ is called \emph{the linear part of $f$}, this map induces an action on the torus that we will also call by $A:\mathbb{T}^d\to \mathbb{T}^d$.

A diffeomorphism $f:\mathbb{T}^n\to\mathbb{T}^n$ is called \emph{partially hyperbolic} if there exists $0<\gamma<1$, $C>0$ and an invariant splitting  $T_x M=E^{uu}(x)\oplus E^c(x)\oplus E^{ss}(x)$, such that
$E^{uu}(x)$ is expanding, $E^{ss}(x)$ is contracting and
\begin{equation}\label{eq.domin}
\begin{aligned}
\norm{Df^{n}(x)v^c}&\leq C\gamma^{n}\norm{Df^{n}(x)v^{uu}}\\\quand \norm{Df^{-n}(x)v^c}&\leq C\gamma^{n}\norm{Df^{-n}(x)v^{ss}},
\end{aligned}
\end{equation}
for every $n\in\natural$, and unitary vectors $v^u\in E^{uu}$, $v^c\in E^c$ and $v^s\in E^{ss}$.

In the literature this is also called \emph{pointwise partial hyperbolic} in contraposition of \emph{absolute partial hyperbolic} when \eqref{eq.domin} is true even comparing the norms in different points $x,y\in M$
(See \cite{BP74}). 

We say that an Anosov map $f$ is a \emph{partially hyperbolic Anosov map} if it is partially hyperbolic and the $\dim E^*\neq 0$ for $*=ss,uu,c$. 

If $f:\torus\to\torus$ and $\dim E^u=2$ we have that 
 $E^u(x)=E^{uu}(x)\oplus E^{c}(x)$, as in this case the center bundle $E^c$ is also expanding we will refer to it as \emph{weak unstable} bundle and denote it as $E^{wu}$, also we refer to $E^{uu}$ as \emph{strong unstable} bundle.
From now on we are going to work in $\torus$, and assume that $\dim E^u=2$ (of course all results can be stated when $\dim E^s=2$ just changing $f$ by its inverse $f^{-1}$).

The stable and unstable distributions of an Anosov map are always integrable in $C^\nu$ foliations $\cW^s$ (contracting) and $\cW^u$ (expanding), also the strong unstable distribution $E^{uu}$ is integrable in an expanding foliation $\cW^{uu}$.
In general for a partially hyperbolic diffeomorphism the center direction is not integrable, if it is integrable we say that $f$ is \emph{dynamically coherent}.

In the case we are considering here (Partially hyperbolic Anosov map in $\torus$) dynamically coherence follows form \cite{Pot15}: in fact as $f$ is an Anosov map then its linear part is a hyperbolic matrix, so it is isotopic to Anosov (see \cite[section~2.3]{Pot15}), also we are assuming that we have a partially hyperbolic map of the type $E^s\oplus E^c\oplus E^u$ (in \cite{Pot15} this is called strong partially hyperbolic) then by \cite[theorem~A.1]{Pot15} $f$ is dynamically coherent. This implies that $\cW^u$ is sub-foliated by two transversal foliations $\cW^{uu}$ and $\cW^{wu}$.

\begin{definition}
 Given an expanding foliation $\cW$ we say that $\lambda_\cW(x)$ is a Lebesgue Lyapunov exponent if $\lambda_\cW$ is the Lyapunov exponent for some $f$-invariant measure absolutely continuous in 
 $\cW$.
\end{definition}
\begin{remark}
It is a clasical result (see \cite{PS82} ) that there exist invariant measures that are absolutely continuous with respect to any expanding foliation. We can apply this result to each one of our invariant foliations (for a simple proof in this case see \cite[step~2]{Gog08}). The measures we obtain that way have full support, this will be proved in section~\ref{s.linear}. 

These measures need not to be the same, for example if $f$ is volume preserving, then the volume measure 
is absolutely continuous with respect to the stable and unstable foliation, but in general is not absolutely continuous with respect to the weak unstable foliation
(see \cite{SX09}).
\end{remark}

Now we can state our main result:
\begin{ltheorem}\label{teo}
 Let $f:\torus\mapsto \torus$ be an Partially Hyperbolic Anosov map, then if there are Lebesgue Lyapunov exponents $\lambda_{\cW^\sigma}=\chi^{\cW^\sigma}_{f}$, for $\sigma=s,uu,wu$, then
 $f$ is $C^1$ conjugated to its linear part.
In particular $f$ preserves some measure absolutely continuous with respect to Lebesgue and all their Lyapunov exponents with respect to this measure are Lebesgue Lyapunov exponents. 
 \end{ltheorem}

 Similar rigidity results were proven by R. de la Llave (\cite{Lla92}) in dimension $2$ and A. Gogolev and M. Guysinsky (\cite{Gog08}) in dimension $3$. In \cite{Lla92} and \cite{Gog08}, regularity of the conjugacy is guaranteed by information on the periodic data. Here we replace by a somewhat weaker condition.

As a corollary of our result we have:
\begin{ltheorem}
If $f:\torus\mapsto \torus$ is a volume preserving partially hyperbolic Anosov map, with absolutely continuous weak unstable foliation and the Lyapunov exponents are equal to the corresponding geometric growth then $f$ is $C^1$ conjugated to its linear part.
\end{ltheorem} 
 
Observe that by \cite{Var14}, the sole absolute continuity of the central foliation does not imply the smoothness of the conjugacy.
 
Another corollary that we can take from the proof of Theorem~\ref{teo} (more specifically this will be a corollary of theorem~\ref{t.cr}),
\begin{ltheorem}\label{teo2}
 Let $f:\torus\mapsto \torus$ be a volume preserving Anosov map (not necessarily partially hyperbolic), then generically $\lambda_{\cW^u}<\chi^u_{f}$, if $E^u$ is one dimensional (or $-\lambda_{\cW^s}<\chi^s_{f^{-1}}$ if not).
 \end{ltheorem}
 
 \subsection{Idea of the proof}
 First we are going to prove that in this case the geometric growth is equal to a topological entropy, this is done in section~\ref{s.entropy}, and this entropy is invariant by conjugation so is the same for the linear part of $f$.
 
To prove the differentiability of the conjugation we will prove the differentiability in each of the invariant foliations and then use Journe's regularity lemma to conclude the regularity of the conjugation.

The differentiability in each foliation will follow from the fact that a map between one dimensional manifolds is $C^1$ if and only if the push forward of the Lebesgue measure gives a measure absolutely continuous to Lebesgue with continuous Radon-Nikodym derivatives. To prove this property we will push the measures absolutely continuous in each invariant manifold and use Ledrappier-Young formula to conclude the absolute continuity of the map $h$ in each direction.

Section~2 is devoted to recall the terminology and results of \cite{Le84a} and \cite{LY85a} that we will use.

\medskip
\textbf{Acknowledgements.} Thanks to Radu Saghin for the useful discussions and recommendations about the work, to Marcelo Viana for the orientation on this problem, to Fernando Lenarduzzi for reading the paper, also to the anonymous referee for the corrections and recommendations on the manuscript.

\section{metric entropy along foliations}\label{s.partitions}
In this section we recall some definitions and results of \cite{Le84a},\cite{Le84b} and \cite{LY85a}, these results are originally stated for the unstable manifolds but they can be easily adapted to any expanding invariant foliation.

Fix some expanding $f$-invariant foliation $\cW$ and some $f$-invariant probability measure $m$.
 
\begin{definition}
Let $\xi$ be a measurable  partition of $M$, let $C_{\xi}\left(x\right)$ be the element of the partition that contains $x\in M$, we said that $\xi$ is \emph{subordinated to $\cW$} if $C_{\xi}\left(x\right)\subset \cW\left(x\right)$, 
for $m$-almost every $x\in M$.
\end{definition}

Given a partition $\xi$ subordinated to $\cW$, by Rokhlin disintegration theorem (see for example~\cite{FET}), we have a family of measures $\{m^{\xi}_{x}\}_{x\in M}$ such that for every measurable set $A\subset M$:
\begin{itemize}
\item $y\mapsto m^{\xi}_{y}$ is constant on each $C_\xi(x)$,
\item $m^{\xi}_{x}\left(C_{\xi}\left(x\right)\right)=1$,
\item $x\mapsto m^\xi_x(A)$ is measurable,
\item $m\left(A\right)=\int m^{\xi}_{x}\left( A \right) dm$.
\end{itemize}

We call a $\cW$ subordinate partition \emph{adapted} if
\begin{itemize}
\item for $m$-almost every $x\in M$, $C_{\xi}\left(x\right)$ is a neighborhood of $x$ in $\cW$, and
\item $f^{-1}\left(\xi\right)\succ \xi$, by this we mean that for $m$-almost every $x\in M$, $f^{-1}C_\xi(f(x))\subset C_\xi(x)$.
\end{itemize}

By \cite[Proposition~3.1]{Le84a} we have that
\begin{lemma}
There exists a $\cW$ adapted partition. 
\end{lemma}
Let $\cH\left(f^{-1}\xi\mid \xi\right)=-\int \log\,m^{\xi}_{x}\left(C_{f^{-1}\xi}(x)\right)$, this does not depend on the adapted subordinated partition (see \cite[Lemma~3.1.2]{LY85a}), so we define the metric $\cW$-entropy as 
$$
\ch_m(f,\cW)=\cH\left(f^{-1}\xi\mid \xi\right).
$$

Let $J^{\cW} f\left(z\right)$ be the Jacobian of $Df(z')\mid_{T\cW}$, for $z'\in \cW(z)$
define $$
\Delta\left(z,z'\right)=\prod_{j=1}^{\infty}\frac{J^{\cW} f\left(f^{-j}z\right)}{J^{\cW} f\left(f^{-j}z'\right)}.
$$
We also recall the next Proposition (\cite[Proposition~3.7]{Le84a}),
\begin{proposition}\label{p.leddens}
Let $\xi$ be an adapted partition, then $\cH\left(f^{-1}\xi \vert \xi\right)=\int \log J^{\cW} f(x) dm$, if and only if, the measure is absolutely continuous in $\cW$, 
moreover 
\begin{equation}\label{eq.abscont}
m^{\xi}_x(B)=\frac{1}{L(x)}\int_{C_\xi (x)\cap B}\Delta(x,t)d \vol_{\cW(x)}(t),
\end{equation}
where $\vol_{\cW(x)}$ is the Lebesgue measure in $\cW(x)$ and $L(y)=\int_{C_\xi (x)}\Delta(x,t)d \mu^u_x(t)$.
\end{proposition}
Observe that $\int \log J^{\cW} f(x) dm=\int \sum_{E^i\in T\cW}\dim E^i\lambda^i_{m} dm$ where $\lambda^i_{m}$ are the Lyapunov exponents of $m$, then
\begin{theorem}[Theorem~4.8 of \cite{Le84a}]\label{t.ledrappier}
Let $\cW$ be an expanding invariant foliation, then $\ch_m(f,\cW)=\int \sum_{E^i\in T\cW}\lambda^i_{m} dm$, if and only if, the disintegration of the measure with respect to $\cW$ is absolutely continuous.
\end{theorem}
\begin{remark}\label{r.rho}
Observe that equation~\eqref{eq.abscont} gives that the disintegration of the measure $m$ with respect to a partition subordinated to the foliation $\cW$ is of the form $dm_x=\rho d\vol_{\cW(x)}$, where $\vol_{\cW(x)}$ is the Lebesgue measure in $\cW(x)$ induced by the restriction of the Riemannian structure of $M$ to $\cW(x)$ and $\rho$ is defined modulo a normalization by the relation 
\begin{equation}\label{eq.delta}
\frac{\rho(x)}{\rho(y)}=\Delta(x,y),
\end{equation}
for every $x,y$ in the same $\cW$ leave.

Also, when $\cW$ is one dimensional, $x\mapsto \vol_{\cW(x)}$ is locally continuous (continuous in local charts) in the weak$^*$ topology.
Lets explain how this follows: by definition of the foliation there exists a foliated chart $\Phi:[-1,1]^d\to M$ centered at $x$, then for every $y\in [-1,1]^{d-1}$, $I\subset [-1,1]\times\{y\}$, 
$$\vol_{\cW(y')}(\Phi(I))=\int_I \norm{\frac{\partial \Phi}{\partial x}(x',y)}dx',$$ where $y'=\Phi(0,y)$ and $\norm{\cdot}$ is the norm induced by the Riemannian metric.

Normalizing $\rho$ such that, for every $y\in [-1,1]^{d-1}$, $\int_{\Phi([-1,1]\times \{y\})} \rho d\vol_{\cW(y')}=1$, by \eqref{eq.delta} we have that $\rho(y')\int \Delta(t,y')d\vol_{\cW(y')}(t)=1$, then 
$y\mapsto \rho d\vol_{\cW(y')}$ is also locally continuous. 
 \end{remark}

\section{Topological entropy along foliations}\label{s.entropy}
In this section we define some topological invariant quantities that are related to the geometric growth, in particular on the case we are treating we get that the geometric growth is topologically invariant.

\begin{definition}
 Given $K\subset \cW(x)$ with $\overline{K}$ compact, we call $E\subset K$, $n,\epsilon$-\emph{separated} if for every $y,z\in E$, $d_{x}^{n}(y,z)>\epsilon$, 
 where $$
 \d_{x}^{n}(y,z)=\max_{0\leq j<n}\d_{\cW}\left(f^{j}(y),f^{j}(z)\right).$$
 
 We say that $E$ is $n,\epsilon$-\emph{generator} if for every $z\in K$ there exist some $y\in E$ such that $\d_{x}^{n}(y,z)<\epsilon$.
\end{definition}
Let $s(K,n,\epsilon)$ be the maximal cardinality of an $n,\epsilon$-separated set, and $g(K,n,\epsilon)$ be the minimum cardinality of an $n,\epsilon$-generator set.
Define 
$$
 s(K,\epsilon)=\limsup\frac{1}{n}\log s(K,n,\epsilon),
$$
$$
 g(K,\epsilon)=\limsup\frac{1}{n}\log g(K,n,\epsilon)
$$
and $$
s(K)=\lim_{\epsilon\rightarrow 0}s(K,\epsilon),\quad g(K)=\lim_{\epsilon\rightarrow0}g(K,\epsilon).
$$

It is easy to see (same argument as in the classical entropy case, see \cite{FET}), that $s(K)=g(K)$, so we define the \emph{$\cW$ entropy} as
$$
 \ch_{\cW}=\sup_{x\in M} \sup_{K\subset \cW(x)} s(K)=\sup_{x\in M} \sup_{K\subset \cW(x)} g(K).
$$ 

We are going to prove that in our case this entropy is equal to the geometric growth.
\begin{lemma}\label{l.entropyineq1}
If $\cW$ is a one dimensional expanding foliation then $\ch_{\cW}\leq \chi_f^\cW$.
\end{lemma}
\begin{proof}
	Let $K\subset \cW(x)$ be a compact set, we have that $\cW(x)$ is homeomorphic to $ \mathbb{R} $ so we can take $\widehat{K}$ to be the smallest convex 
	interval that contains $K$. 
	
Observe that for every $n>0$, $f^{n}\widehat{K}$ is an interval on $\cW(f^n(x))$, and has at most $\frac{\leng\, f^{n}\widehat{K}}{r}$ $r$-separated points, where $\leng(I)$ is the length of the interval 	$I$ in $\cW$.	
	
	Take $k$ such that $\norm{Df^{-k}\vert_{T_x\cW}} < 1$, then 
for every $x,y\in \widehat{K}$, $\d_\cW(x,y)\leq \d(f^k(x),f^k(y))$.
So 
$$\d^n_x(y,z)=\max_{n-k\leq j< n}\d_\cW(f^j(y),f^j(z)),$$
then if $E$ is an $n , r$-separated set of maximal cardinality we have
	$$
	\# E\leq\, \frac{\leng\, f^{n}\widehat{K}}{r}+\frac{\leng\, f^{n-1}\widehat{K}}{r}+\cdots +\frac{\leng\, f^{n-k}\widehat{K}}{r}.$$

	So we conclude that
	$$
	\limsup_{n\to \infty}\frac{1}{n}\log\,s\left(K,n,r\right)\leq \limsup_{n\to \infty}\frac{1}{n}\log\left(\frac{\leng\, f^{n}\widehat{K}}{r}\right)\leq \chi_f^\cW.
	$$
	\end{proof}

\begin{lemma}\label{l.entropyineq}
Let $\cW$ be a $k$-dimensional manifold (not necessarily one dimensional) then $\ch_{\cW}\geq \chi_f^\cW$.
\end{lemma}
\begin{proof}
Fix $y\in M$ and $R,r\in \mathbb{R}$ and let $E$ be and $n,r$-generator set of $\cW_R(y)$ with minimal cardinality, we have that 
$$
f^{n}\left(\cW_R(y)\right)\subset \bigcup_{x\in E} \cW_r(f^n(x)) ,
$$
Hence, we have that
$$
\begin{aligned}
\vol\,f^{n} \left( \cW_R(y) \right) & \leq  \sum_{x\in E} \vol \cW_r(f^{n}(x))\\
 &\leq  C\#E r^{k}
\end{aligned}
 $$
 where $C>0$ is some constant that depends on the curvature of $\cW$.
 
So we conclude that 
$$
\limsup_{n\to \infty}\frac{1}{n}\log\left(\vol f^{n}(\cW_R(y))\right)\leq \limsup_{n\to \infty}\frac{1}{n}\log\left(r^{k}\, g\left( \cW_R(y),n,r \right) \right).
$$

\end{proof}
As a consequence of lemmas \ref{l.entropyineq1} and \ref{l.entropyineq} we have
\begin{proposition}\label{p.entropyGG}
Let $\cW$ be a one-dimensional $f$-invariant expanding foliation, then $\ch_{\cW}= \chi_f^\cW$.
\end{proposition} 
\section{The linear part} \label{s.linear}
Let $f:\torus\to\torus$ be an Anosov map as in the hypothesis of theorem~\ref{teo}. Let $F:\real^3\to\real^3$ be the lift of $f$ to $\real^3$ and $A\in \SL$ be its linear part, this means that there exists some $p:\real^3\to\real^3$, with $p(x+n)=p(x)$ for every 
$x\in\real^3$ and $n\in\integer^3$, such that $F=A+p$.
 
By a classical result of A. Manning (\cite{Mann74}) the linear part of a partially hyperbolic Anosov map, $A$, has $3$ different eigenvalues $\alpha^{uu}$, $\alpha^{wu}$ and $\alpha^s$ with $\lambda^{uu}_A> \lambda^{wu}_A>0>\lambda^{s}_A$, where $\lambda_A^i=\log \abs{\alpha^i}$ for $i=uu,wu,s$.

 Also, a classical result (see for example \cite{KaH95}) states that there exists a conjugacy 
$$
H:\real^3\to\real^3,\quad\text{ such that }H\circ F=A\circ H,$$
 at finite distance of the identity that also induce a conjugacy 
 $$
 h:\torus\to\torus,\quad h\circ f=A\circ h.
 $$

Observe that $A:\torus\to \torus$ is a partially hyperbolic Anosov map with invariant manifolds $\cW^i_A(x)=\{x\}+E^i_A$ where $E^i_A$ is the eigenspace corresponding to $\lambda^i_A$, for $i=uu,wu,s$.

By \cite[proposition~2.3]{Pot14} we have that:
\begin{itemize}
\item $h(\cW^s_f(x))=\cW^s_A(h(x))$,
\item $h(\cW^{uu}_f(x))\subset \cW^u_A(h(x)):=\{x\}+E^{uu}_A+E^{wu}_A$,
\item $h(\cW^{wu}_f(x)=\cW^{wu}_A(h(x))$,
\item $h(\cW^u_f(x))=\cW^u_A(h(x))$.
\end{itemize}

The following results are going to be used for each of the invariant foliations so from now on lets suppose that $\cW$ is a one dimensional $f$-invariant foliation and $h$ takes the leaves of $\cW$ to $A$-invariant manifold $E_A$ (a translation of an eigenspace). 

The purpose of this section is to prove:
\begin{theorem}\label{t.cr}
 If  $h$ takes any one dimensional expanding $f$-invariant $C^\nu$ foliation $\cW$ (for example $\cW^{uu}$ or $\cW^{wu}$) to an $A$-invariant manifold $E_A$ and there exists an 
 $f$-invariant measure $m$ absolutely continuous on $\cW$ such that $\chi^\cW_f=\lambda^m_{\cW}(x)$, $m$-almost everywhere, then for every $x\in M$, $h$ is $C^\nu$ restricted to $\cW(x)$.
\end{theorem}

Observe that, as $E_A$ is dense in $\torus$ and $h$ is a homeomorphism, $\cW(x)$ is also a dense for any $x\in M$, a foliation with this property is called minimal. 

The support of a measure $m$ absolutely continuous on $\cW$ is saturated by $\cW$, then by the minimality of this foliation and proposition~\ref{p.leddens} the measure is fully supported. 

Now, lets state some lemmas in order to prove the theorem.
\begin{lemma}\label{l.entropyinv}
We have that $\ch_{\cW}^{f}=\ch_{E_A}^{A}$.
\end{lemma}
\begin{proof}
The foliation $\cW$ is a continuous foliation with $C^{1}$ leaves, so there exist $\delta > 0$ and $Q>0$ such that if 
$\d(x,y)<\delta$, with $y\in \cW_{loc}(x)$, then $\d_{\cW}\left(x,y\right)\leq Q \d\left(x,y\right)$.

By the uniform continuity of $h$, given $K\subset \cW(x)$ compact and $\varepsilon>0$, there exists $\delta>0$ such that, for every $x,y\in K$, if $\d_{\cW_{f}}(x,y)<\delta$ then $\d_{E_A}\left(h(x),h(y)\right)<\varepsilon$. 
We claim that if $E$ is a $n,\delta$-generator set of $K$ then $h(E)$ is a $n,\varepsilon$-generator set of $h(K)$.

Indeed, for every $y\in K$ there exist $x\in E$ such that $\d_{\cW}(f^{j}(x),f^{j}(y))<\delta$, we have that $\d_{E_A}\left(h\circ f^{j}(x),h\circ f^{j}(y)\right)<\varepsilon$ then
$\d_{E_A}\left(A^{j}h(x),A^{j}h(y)\right)<\varepsilon$, proving our claim.

This implies that $g_{A}\left(h(K),n,\varepsilon \right)\leq g_{f}\left(K,n,\delta\right)$ so $g_{A}\left(h(K),\varepsilon\right)\leq g_{f}\left(K,\delta\right)$, observe that if 
$\varepsilon \rightarrow 0$ then $\delta \rightarrow 0$, also $h$ is a homeomorphism so taking the supremum for every compact set contained in $E_A$ is the same of taking the supremum over $h(K)\subset E_A$ with $K\subset \cW$ compact.
 
Then we conclude that $\ch_{E_A}^{A}\leq \ch_{\cW}^{f}$.
The same reasoning using $h^{-1}$ gives $\ch_{E_A}^{A} \geq \ch_{\cW_A}^{f}$.
\end{proof}

As a corollary we have the next result interesting on itself (for this result we do not need the torus to be $3$-dimensional, neither the invariant manifolds to be one dimensional).

\begin{theorem}
Let $f:\mathbb{T}^n\to \mathbb{T}^n$ be a Partially Hyperbolic Anosov map, then for any measure absolutely continuous in the weak unstable direction, the Lyapunov exponents in this direction are not bigger than the corresponding of the linear part.
\end{theorem}
\begin{proof}
Let $\mu$ be an invariant measure such that $\cW^{wu}$ is absolutely continuous. By \cite{SX09}, $\lambda^\mu_{\cW^{wu}}(x)\leq \chi^f_{\cW^{wu}}$, also by lemma~\ref{l.entropyineq} and lemma~\ref{l.entropyinv} we have that $\lambda^\mu_{\cW^{wu}}(x)\leq  \ch^A_{E^{wu}_A}$.

As $E^{wu}_A$ is an invariant linear subspace $\ch^A_{E^{wu}_A}=\chi^{wu}_A=\log \det A\mid_{E^{wu}_A}$, this is equal to the sum of the logarithm of the eigenvalues of $A$ in the direction $E^{wu}$.
\end{proof}

When $E^{wu}$ is one dimensional, by proposition~\ref{p.entropyGG} we have that $\chi^{E_A}_{A}=\chi^{\cW}_{f}=:\chi$ and, because $A$ is linear and $E_A$ is an one-dimensional eigenspace, the Lyapunov exponent of $A$, the geometric growth and the logarithm of the eigenvalue coincide.

From now on we are going to assume that $\chi_f^\cW=\lambda_\cW^m$, so we have
\begin{equation}\label{eq.equality}
\chi=\lambda_{E_A}=\lambda_\cW^m.
\end{equation}

Using an $f$ subordinated partition $\xi$, as defined section~\ref{s.partitions}, we can define a partition $h\xi$, by 
$$
h\xi(x)=h\left(\xi(h^{-1}(x))\right),$$ 
this 
partition has the following properties:

\begin{lemma}
If $\xi$ is an adapted $f$-subordinated measurable partition of $\cW_{f}$ then $h\xi$ is an adapted $A$-subordinated measurable partition of $E_{A}$ and the Rokhlin 
disintegration of $\hat{m}:=h_{\ast}m$ is $\hat{m}^{h\xi}_{x}=h_{\ast}m^{\xi}_{h^{-1}(x)}$.
\end{lemma} 
\begin{proof}
It is clear that $h\left(\xi(h^{-1}(x))\right)\in h\left(\cW(h^{-1}(x)\right)=E_{A}(x)$, and 
\begin{eqnarray*}
A^{-1}\left(h\xi\right)\left(Ax\right)&=&\left(A^{-1}h\right)\xi\left(h^{-1}Ax\right)\\
&=& h\left(f^{-1}\xi\left(f\,h^{-1}(x)\right)\right) \\
&\subset& h\left(\xi\left(h^{-1}(x)\right)\right),
\end{eqnarray*}
also open neighborhoods goes to open neighborhoods because $h$ is a homeomorphism.

Now for the Rokhlin disintegration, observe that
\begin{equation*}
h_{\ast}m^{\xi}_{h^{-1}(x)}\left(h(\xi(h^{-1}(x))\right)=m^{\xi}_{h^{-1}(x)}\left(\xi(h^{-1}(x))\right)=1,
\end{equation*}
also given any mensurable set $B$
\begin{eqnarray*}
\int h_{\ast}m^{\xi}_{h^{-1}(x)}(B)d(h_{\ast}m) &=& \int m^{\xi}_{h^{-1}(x)}\left(h^{-1}\left(B\right)\right)d(h_{\ast}m)\\
 &=& \int m^{\xi}_{x}\left(h^{-1}\left(B\right)\right)d m\\
 &=& m\left(h^{-1}B\right)\\
 &=& h_{\ast}m\left(B\right). 
\end{eqnarray*}

So the result follows by the uniqueness of the Rokhlin disintegration
\end{proof}

We have that the metric entropy in the $E_A$ manifold for $h_{\ast}m$ is 
\begin{eqnarray*}
\ch_{h_\ast m}^{A}&=& -\int \log \hat{m}^{h\xi}\left(A^{-1}(h\xi)\right)d(h_{\ast} m)\\
 &=& -\int \log h_\ast m^{\xi}\left(A^{-1}\left(h\xi\left(h^{-1}x\right)\right)\right)d(h_\ast m)\\
&=& -\int \log m^{\xi}\left(f^{-1}\xi(x)\right) dm \\
&=&\chi,\\
\end{eqnarray*}
the last equality follows by the fact that $m$ is absolutely continuous in the $\cW$ direction, theorem~\ref{t.ledrappier} and \eqref{eq.equality}.

Then $h_{\ast}m=\lambda_{E_A}$ so applying Proposition~\ref{p.leddens}, the measure $h_{\ast} m$ is absolutely continuous in the $E_A$ direction with $\Delta_A(y,t)=1$ (because the Jacobian is constant for the linear map).

We need to prove that, restricted to $\cW$, $h$ preserves the density $\rho$, and for every 
$t\in\cW(y)$, $\Delta_{f}(y,t)=\prod \frac{J^{\cW} f(f^{-j}(y))}{J^{\cW} f(f^{-j}(t))}$ is continuous.
 
Actually if $f$ and $\cW$ are $C^{r}$ then $\Delta_{f}(y,t)=\prod \frac{J^{\cW} f(f^{-j}(y))}{J^{\cW} f(f^{-j}(t))}$ is $C^{r-1}$  
with $r>1$, for $r=1$ we need $f$ to be at least $C^{1+\alpha}$, for some $\alpha>0$.  

The proofs of the next lemmas can be found in \cite{Lla92}, for completeness we redo the proofs here.

\begin{lemma}\label{l.cr}
  $\Delta_{f}(y,t)=\prod \frac{J^{\cW} f(f^{-j}(y))}{J^{\cW} f(f^{-j}(t))}$ is $C^{r-1}$ function in the $\cW$ direction if $f$ is $C^{r}$ with $r>1$, for
  $r=1$ is $C^{0}$ if $f$ is $C^{1+\alpha}$.
\end{lemma}
\begin{proof}
 Let  
 \begin{equation*}
  \log\Delta_{n}(y,t)=\sum_{j=0}^{n-1}\left( \log J^{\cW} f(f^{-j}(y))-\log J^{\cW} f(f^{-j}(t))\right).
 \end{equation*}
  If $f$ is $C^{1+\alpha}$ then $J^{u}(x)$ is $\alpha$ H\"older,
  hence 
  \begin{equation*}
  \left( \log J^{\cW} f(f^{-j}(y)-\log J^{\cW} f(f^{-j}(t)\right)\leq C \lambda^{-j\alpha} d(x,y)^{\alpha}
 \end{equation*}
 So, $\Delta_n$ converges absolutely, this implies the continuity of $\Delta$.
 
 If $r>1$ then $\log J^{\cW} f(y)$ is $C^{r-1}$ restricted to $\cW$, and
 \begin{equation*}
  \frac{dJ^{\cW}f\circ f^{-j}}{d\cW}=D\log J^{\cW}f(f^{-j}(x))Df^{-j}(x)\mid_\cW,
 \end{equation*}
 so 
   \begin{equation}\label{eq.deriv}
  \norm{\frac{d\log J^{\cW}f \circ f^{-j}}{d\cW}(x)}\leq C\lambda^{-j}\norm{v^{u}}.
 \end{equation}
Then
 $$
\frac{d}{dy}\Delta_n(y,t)=\sum_{k=0}^{n-1} \frac{1}{\log J^{\cW}f(f^{-k}(t))}\prod_{j\neq k} \frac{\log J^{\cW}f(f^{-j}(y))}{\log J^{\cW}f(f^{-j}(t))}\frac{dJ^{\cW}f\circ f^{-k}}{d\cW} 
 $$
 by the first part $\prod_{j\neq k} \frac{\log J^{\cW}f(f^{-j}(y))}{\log J^{\cW}f(f^{-j}(t))}$ is uniformly bounded, so using~\eqref{eq.deriv} we conclude that the derivative also converges uniformly.
 Similar arguments for the derivatives of superior order shows that all the derivatives converges uniformly.
\end{proof}

\begin{proposition}\label{p.abscont}
For $m$ almost every $x$ we have that $h:\cW_{f}(x)\rightarrow E_{A}(h(x))$ is $C^{r}$.
\end{proposition}
\begin{proof}
Fix $x$ such that the disintegration of the measure in a $\cW$ adapted partition is absolutely continuous to Lebesgue in $\cW$, we have $h:C_{\xi}(x)\to C_{h\xi}(h(x))$. 
Using a foliated chart, we can assume that $h:I\rightarrow J$ where $I$ and $J$ are intervals.

We have that the measure induced in the intervals by our invariant measure is absolutely continuous, this means
\begin{equation*}
\int \varphi d(h_{\ast}m^{\xi})=\int \varphi \rho_{1} d \vol_J\quand \int \varphi d(m^{\xi})=\int \varphi \rho_{2} d\vol_I,
\end{equation*}
remember that as $A$ is linear $\Delta_A(x,y)=1$ and then $\rho_1$ is constant. 

Then
\begin{eqnarray*}
\int \left(\varphi h\right)\rho_{2} d\,\vol &=& \int \varphi d(h_{\ast}m^{\xi})\\
&=&\int \varphi \rho_{1} d\,\vol.
\end{eqnarray*}
So we have that $h$ is given by
\begin{eqnarray*}
h(y)-h(x)&=&\int_{h(x)}^{h(y)}1 d\,\vol\\
&=&\int 1_{\left[h(x),h(y)\right]}\frac{\rho_{1}}{\rho_{1}} d\,\vol\\
&=&\int 1_{\left[h(x),h(y)\right]}\circ h \frac{\rho_{2}}{\rho_{1}} d\,\vol\\
&=&\int 1_{\left[x,y\right]}\frac{\rho_{2}}{\rho_{1}} d\,\vol, 
\end{eqnarray*}
where $1_B$ is the characteristic function of $B$.

Hence, because $\rho_{2}$ is $C^{r-1}$ by lemma~\ref{l.cr}, we have that $h\in C^{r}$.
\end{proof}

\begin{proof}[proof of theorem~\ref{t.cr}]
 Fix a foliated chart $\Phi:[-1,1]^d\to M$ centered at $x$,
 also fix the partition $\varsigma$ into horizontal discs given by this foliation.
 
 Define $\rho:\Phi([-1,1]^d)\to \real$ such that 
 $$\frac{\rho(\Phi(s))}{\rho(\Phi(t))}=\Delta_{f}(\Phi(s),\Phi(t))\quand\int_{\Phi([-1,1]\times \{y\})} \rho d \vol_{\cW(\Phi(0,y))}=1$$ for every $y\in [-1,1]^{d-1}$, $t,s\in [-1,1]\times \{y\}$.
 
 If $x\in \supp(m)$ then we can take a sequence of $x_{n}\rightarrow x$ such that $m^{\varsigma}_{x_{n}}$ and $h_{\ast}m^{\varsigma}_{x_{n}}$ are absolutely continuous with
 respect to Lebesgue.
 We claim that $m^{\varsigma}_{x_{n}}\rightharpoonup \rho d\vol_{\cW(x)}$.
 To prove this claim take any $g:\Phi([-1,1]^d)\to \real$ continuous. Write $g$ in the foliated chart as $g(t,z)$ with $t\in [-1,1]$ and $z\in [-1,1]^{d-1}$, so by assumption 
 \begin{equation*}
  \int g(t,z) dm_{x_{n}}=\int g(t,x_{n})\rho(t,x_n) d \vol_{\cW(x_n)}(t)
 \end{equation*}
then by remark~\ref{r.rho} when $n\to \infty$, we have that
 \begin{equation*}
  \int g(t,z)\,dm_{x_{n}}\rightarrow \int g(t,x) \rho(t,x) d\vol_{\cW(x)}(t).
 \end{equation*}
 
Also, because $h$ is continuous, $h_{\ast}m^{\varsigma}_{x_{n}}\rightharpoonup h_{\ast}\left(\rho d\vol_{\cW(x)}\right)$ and by the same arguments as before  
we also have that $h_{\ast}m^{\varsigma}_{x_{n}}\rightharpoonup \rho_{A} d\vol_{E_A}$, so $h$ also takes Lebesgue measure with $C^{r-1}$ density to Lebesgue measure with $C^{r-1}$ density in $\cW(x)$,
then we can apply the same argument as in proposition~\ref{p.abscont} to every $x\in \supp(m)=\torus$.
\end{proof}

\section{Proof of the Main Theorem}

Before finishing the proof we recall the next lemma due to Journ\'e (\cite{Jo88})
\begin{lemma}\label{l.regularity}
Let $M_j$ be a manifold and $\cW^s_j$, $\cW^u_j$
be continuous transverse foliations with $C^\nu$, $\nu>0$, leaves, for $j= 1,2$. Suppose that $h:M_1\to M_2$ is a homeomorphism that maps
$\cW^s_1$ into $\cW^s_2$ and $\cW^u_1$ to $\cW^u_2$. Moreover assume
that the restrictions of
$h$
to the leaves of these foliations are
$C^\nu$, then $h$ is $C^\nu$.
\end{lemma}
Now we are able to prove our main theorem.
\begin{proof}[Proof of Theorem~\ref{teo}]
Fix some invariant measure $m$ absolutely continuous on $\cW^{wu}$ such that $\lambda^m_{wu}(x)=\chi_{wu}$. 

As $h$ preserves the weak unstable manifold, by theorem~\ref{t.cr} we have that 
$h$ is $C^1$ in the weak unstable direction. Using $f^{-1}$ and the fact that $h$ preserves the stable manifold we also conclude that $h$ is $C^1$ in the stable direction.

In general $h(\cW^{uu})$ is not necessarily contained in $E^{uu}_A$, but in \cite[lemma~6]{Gog08} Gogolev and Guysinsky proved that if $h$ restricted to the weak unstable manifold is $C^1$ (as in our case), then $h(\cW^{uu})=E^{uu}_A$. 

Applying theorem~\ref{t.cr} once more we conclude that $h$ is also 
$C^1$ in the strong unstable direction.

Using lemma~\ref{l.regularity}, first with the $\cW^{uu}$ and $\cW^{wu}$, sub-foliations of $\cW^u$, we conclude that $h$ is $C^1$ in $\cW^u$, using again the lemma with $\cW^u$ and $\cW^s$ we conclude that $h:\torus\to \torus$ is $C^1$.

For the second part observe that, as $A$ preserves the volume measure $\vol$ in $\torus$ and $h$ is $C^1$, we have that $\mu=h^{-1}_* m$ is absolutely continuous with $\vol$. Moreover, as $h^{-1}(E^{wu}_A)=\cW^{wu}$ we have that $\cW^{wu}$
is absolutely continuous, so the Lyapunov exponents with respect to $\mu$ are Lebesgue Lyapunov exponents in every direction.
\end{proof}
\begin{remark}
The weak unstable manifold is only $C^1$, this is the reason why we can not prove that $h$ is actually $C^\nu$, we only have that $h$ is $C^\nu$ on the strong unstable and stable directions.
\end{remark}

 \begin{proof}[Proof of Theorem~\ref{teo2}]
 Let $m$ be the volume measure. Suppose that $E^u$ is one dimensional and $\lambda^u_m=\chi^u_f$, as $f$ is volume preserving, $\cW^u$ is absolutely continuous and $h(\cW^u)=E^u_A$  we can apply theorem~\ref{t.cr} to $\cW^u$, then the conjugation between $f$ and $A$ is $C^1$ in the $u$-direction. In particular there exist $\lambda>0$ such that
 $ \norm{Df^{k_p}(p)\mid_{E^u(p)}}=e^{k_p \lambda}$ for every $k_p$-periodic point $p$. 
 Take two periodic points $p$ and $q$ and make a perturbation such that 
 the eigenvalue in the unstable direction in $p$ changes but the one in $q$ does not. By theorem~\ref{t.cr} this implies that $\lambda^u<\chi$, observe that the condition 
 $ Df^{k_p}(p)\mid_{E^u(p)}=e^{k_p \lambda}$ and $ Df^{k_q}(q)\mid_{E^u(q)}=e^{k_q \lambda'}$ with $\lambda\neq \lambda'$ is an open condition, then we have that $\lambda^u<\chi$
 in an open and dense set of the volume preserving Anosov maps.
 \end{proof}

\medskip{\bf Acknowledgements.} The author was partially supported by Universit\'e Paris 13.


\begin{thebibliography}{10}

\bibitem{Var14}
R.~Var\ ao.
\newblock Center foliation: absolute continuity, disintegration and rigidity.
\newblock {\em Ergod. Th. \& Dynam. Sys.}

\bibitem{BP74}
M.~Brin and Ya. Pesin.
\newblock Partially hyperbolic dynamical systems.
\newblock {\em Izv. Acad. Nauk. SSSR}, 1:177--212, 1974.

\bibitem{Lla92}
R.~de~la Llave.
\newblock Smooth conjugacy and s-r-b measures for uniformly and non-uniformly
  hyperbolic systems.
\newblock {\em Comm. Math. Phys.}, 150(2):289--320, 1992.

\bibitem{Gog08}
Andrey Gogolev and Misha Guysinsky.
\newblock $c^1$-differentiable conjugacy of anosov diffeomorphisms on three
  dimensional torus.
\newblock {\em Discrete and Continuous Dynamical Systems}, 22(1/2):183--200,
  2008.

\bibitem{Jo88}
J.-L. Journ{\'e}.
\newblock A regularity lemma for functions of several variables.
\newblock {\em Rev. Mat. Iberoamericana}, 4:187--193, 1988.

\bibitem{KaH95}
A.~Katok and B.~Hasselblatt.
\newblock {\em Introduction to the modern theory of dynamical systems},
  volume~54 of {\em Encyclopedia of Mathematics and its Applications}.
\newblock Cambridge University Press, 1995.
\newblock With a supplementary chapter by Katok and Leonardo Mendoza.

\bibitem{Le84a}
F.~Ledrappier.
\newblock Propri{\'e}t{\'e}s ergodiques des mesures de {S}ina{\"\i}.
\newblock {\em Publ. Math. I.H.E.S.}, 59:163--188, 1984.

\bibitem{Le84b}
F.~Ledrappier.
\newblock Quelques propri\'et\'es des exposants caract\'eristiques.
\newblock {\em Lect. Notes in Math.}, 1097:305--396, 1984.

\bibitem{LY85a}
F.~Ledrappier and L.-S. Young.
\newblock The metric entropy of diffeomorphisms. {I}. {C}haracterization of
  measures satisfying {P}esin's entropy formula.
\newblock {\em Ann. of Math.}, 122:509--539, 1985.

\bibitem{Mann74}
A.~Manning.
\newblock There are no new {A}nosov diffeomorphisms on tori.
\newblock {\em Amer. J. Math.}, 96:422--429, 1974.

\bibitem{Ose68}
V.~I. Oseledets.
\newblock A multiplicative ergodic theorem: {L}yapunov characteristic numbers
  for dynamical systems.
\newblock {\em Trans. Moscow Math. Soc.}, 19:197--231, 1968.

\bibitem{PS82}
Ya. Pesin and Ya. Sinai.
\newblock {G}ibbs measures for partially hyperbolic attractors.
\newblock {\em Ergod. Th. {\&} Dynam. Sys.}, 2:417--438, 1982.

\bibitem{Pot15}
R.~Potrie.
\newblock Partial hyperbolicity and foliations in $\mathbb{T}^3$.
\newblock {\em J. Modern Dynamics}, 9:81--121, 2015.

\bibitem{Pot14}
Rafael Potrie.
\newblock A few remarks on partially hyperbolic diffeomorphisms of
  $\mathbb{T}^3$ isotopic to anosov.
\newblock {\em Journal of Dynamics and Differential Equations}, 26(3):805--815,
  2014.

\bibitem{SX09}
R.~Saghin and Z.~Xia.
\newblock Geometric expansion, {L}yapunov exponents and foliations.
\newblock {\em Ann. Inst. H. Poincar\'e Anal. Non Lin\'eaire}, 26:689--704,
  2009.

\bibitem{FET}
M.~Viana and K.~Oliveira.
\newblock {\em Foundations of Ergodic Theory}.
\newblock Cambridge University Press, 2015.

\end{thebibliography}
\end{document}